\documentclass{amsart}
\usepackage[T1]{fontenc}
\newtheorem{theorem}{Theorem}
\usepackage{amssymb}
\usepackage{amsmath}
\usepackage{styles_1}
\usepackage{dblfloatfix}
\usepackage{hyperref}

\usepackage[numbers,sort&compress]{natbib}
\calclayout
\usepackage{lipsum}
\newcommand{\notimplies}{%
\mathrel{{\ooalign{\hidewidth$\not\phantom{=}$\hidewidth\cr$\implies$}}}}

\title{Symmetry of concentration and scaling for self-bounding functions}
\author{George Crowley and I\~naki Esnaola
}
\date{\today}

\begin{document}

\begin{abstract}
We prove generalised concentration inequalities for a class of scaled self-bounding functions of independent random variables, referred to as ${(M,a,b)}$ self-bounding. The scaling refers to the fact that the component-wise difference is upper bounded by an arbitrary positive real number $M$ instead of the case $M=1$ previously considered in the literature. Using the entropy method, we derive symmetric bounds for both the upper and lower tails, and study the tightness of the proposed bounds. Our results improve existing bounds for functions that satisfy the ($a,b$) self-bounding property.
\end{abstract} 

\maketitle
\section{Keywords}
concentration inequalities; $(a,b)$ self-bounding; ${(M,a,b)}$ self-bounding; independent random variables

\section{Introduction} \vspace{2mm}

The concentration of measure phenomenon for functions of independent random variables has been studied extensively \cite{talagrandConcentrationMeasureIsoperimetric1995,talagrandNewConcentrationInequalities1996}. Classical results such as McDiarmid's bounded differences inequality \cite{mcdiarmidMethodBoundedDifferences1989} and extensions \cite{kontorovichConcentrationUnboundedMetric2013} have proved useful in many settings for bounding the difference between a function and its expected value. More recently, enabled by the application of the entropy method \cite{ledouxTalagrandsDeviationInequalities1997}, a popular class of concentration results has emerged based on functions which satisfy either the $(a,b)$ self-bounding property \cite{boucheronConcentrationInequalitiesNonasymptotic2013,boucheronConcentrationSelfboundingFunctions2009,maurerConcentrationInequalitiesFunctions2006,mcdiarmidConcentrationSelfboundingFunctions2006}, the weakly $(a,b)$ self-bounding property \cite{maurerConcentrationInequalitiesFunctions2006,maurerDominatedConcentration2010}, or their variations (e.g. self-bounding \cite{boucheronSharpConcentrationInequality2000}). A general summary of the most recent results is as follows: for $(a,b)$ self-bounding functions, the most recent bounds for the upper-tail, shown in Theorem \ref{theorem:old_results_upper_tail} are described in \cite{boucheronConcentrationSelfboundingFunctions2009}, whilst the upper bound for the lower tail, shown in Theorem \ref{theorem:old_results_lower_tail}, is studied in \cite{mcdiarmidConcentrationSelfboundingFunctions2006}. For a survey of other similar results, see \cite{boucheronConcentrationInequalitiesNonasymptotic2013}. 

As motivation for this work, it is well known that non-negative monotone submodular functions satisfy the self-bounding property ($a=1,b=0$ of Definition \ref{def:(a,b)self-bounding}) \cite{vondrak2010noteconcentrationsubmodularfunctions}, and satisfy Definition \ref{def:(a,b)self-bounding} further assuming $0 \leq f(\xv) - f_{i}(\xv^{(i)}) \leq 1$ for all $i$. However, there exist examples of these functions in which either $0 \leq f(\xv) - f_{i}(\xv^{(i)}) \ll 1$, or the bound does not hold. Of course, scaling the difference is a viable way of overcoming this issue, but that often comes at the expense of losing tightness due to its effect on the $b$ term. One such example comes from \cite{crowleySubmodularityMutualInformation,crowleyInformationtheoreticSensorPlacement2024} with application to the sensor placement problem, in which the bound explicitly relies on some function of the largest eigenvalue of a covariance matrix. Another example of self-bounding functions includes Rademacher averages \cite{boucheronSharpConcentrationInequality2000,pellegrina2021sharperconvergenceboundsmonte}, with many more examples given in \cite{boucheronConcentrationInequalitiesNonasymptotic2013}. As a result, having a generalised concentration inequality for self-bounding functions will lead to more versatile and accurate results with application in a multitude of science and engineering disciplines.

The contributions of this paper are as follows:
\begin{enumerate}
    \item We modify the definition of ($a,b$) self-bounding functions to ($M,a,b$) self-bounding functions and show that for these functions, there exist symmetric bounds for both the upper and lower tail. The new lower tail bound is a result of using Theorem \ref{theorem:Harris_inequality}.
    \item We show that for values of $M < 1$, Theorem \ref{theorem:M_symmetry} improves upon the existing bounds with $M=1$ in the ($a,b$) self-bounding definition. These results are further refined by Theorem \ref{theorem:M_improved}.
    \item Theorem \ref{theorem:M_improved} describes the conditions on $a$ and $M$ for which Theorem \ref{theorem:M_symmetry} can be strengthed using a similar analysis to that in \cite{boucheronConcentrationSelfboundingFunctions2009}. Moreover, we show in Section \ref{section:scaling} that the upper tail bound result provides improved, or equivalent, concentration in almost all cases when compared to rescaling a ($M,a,b$) function to an ($a,b$) self-bounding function from Theorem \ref{theorem:old_results_upper_tail}. Remark \ref{remark:improvement_upper_tail} discusses the potential for further extension of Theorem \ref{theorem:M_improved} through the application of numerical methods. The lower tail bound in Theorem \ref{theorem:M_improved} always provides stronger concentration than rescaling an ($M,a,b$) function to an ($a,b$) self-bounding function from Theorem \ref{theorem:old_results_lower_tail}.
\end{enumerate}

\subsection{Notation, definitions and previous results} 

Throughout, we assume the random variables $X_1,\dots,X_n$ are independent, with $X_i$ taking values in some space $\mathcal{X}_i$. Let $\Xv \coloneq (X_1,\dots,X_n)$ be the vector taking values in $\mathcal{X}^n \coloneq \prod_{i=1}^n \mathcal{X}_i$, and $\mathcal{F}$ be the associated $\sigma$-algebra on $\mathcal{X}^n$. Define $\mathcal{X}^n_{-i} \coloneq \mathcal{X}_1 \times \dots \times \mathcal{X}_{i-1} \times \mathcal{X}_{i+1} \times \dots \times \mathcal{X}_n$ and $\mathcal{F}_{-i}$ the associated $\sigma$-algebra. Let $\xv$ be a realisation of $\Xv$, and let similarly
    \begin{align}
        \xv^{(i)} \coloneq (x_1,x_2,\dots,x_{i-1},x_{i+1},\dots,x_n)
    \end{align}
be a realisation of $\xv$ with the $i$-th component removed. We assume throughout the rest of the paper that the function $f: \mathcal{X}^n \to \mathbb{R}$ is $\mathcal{F}$-measurable. Moreover, we define the function $f_i : \mathcal{X}^n_{-i} \to \mathbb{R}$ for each $i = 1,\dots,n$ as
        \begin{align} \label{eq:inf}
            f_{i}(\xv^{(i)}) \coloneq \inf_{x_i' \in \mathcal{X}_i} f(x_1,\dots,x_{i-1},x_i',x_{i+1},\dots,x_n).
        \end{align}
We also assume that $f_{i}(\xv^{(i)})$ is $\mathcal{F}_{-i}$ measurable similarly to \cite[Lemma 5]{mcdiarmidConcentrationSelfboundingFunctions2006}. Therein, it is shown the existence of a measurable function $\Tilde{f}_{i}(\xv)$, which does not depend on $x_i$, such that $f_i(\xv^{(i)}) \leq \Tilde{f}_{i}(\xv) \leq f(\xv)$ upto sets of measure zero. We proceed by introducing relevant definitions.      
                
\begin{definition}[$(a,b)$ self-bounding function] \label{def:(a,b)self-bounding}
Let $f: \mathcal{X}^n \to \mathbb{R}$ and  $a,b$ $\geq 0$. The function $f$ is $(a,b)$ self-bounding
if for $\xv \in \mathcal{X}^n$ and $i=1, \dots, n$, it holds that
        \begin{align} \nonumber
                0 \leq f(\xv) - f_{i}(\xv^{(i)}) \leq 1, \end{align} and
        \begin{align} \nonumber
                \sum_{i=1}^{n} (f(\xv) - f_{i}(\xv^{(i)})) \leq a f(\xv) + b.
        \end{align}            
\end{definition}

\begin{definition}[${(M,a,b)}$ self-bounding function] \label{def:(M,a,b)self-bounding}
Let $f: \mathcal{X}^n \to \mathbb{R}$ and  $a,b$ $\geq 0$. The function $f$ is $(M,a,b)$ self-bounding
if for $\xv \in \mathcal{X}^n$ and $i=1, \dots, n$, it holds that
            \begin{align} \nonumber
                0 \leq f(\xv) - f_{i}(\xv^{(i)}) \leq M, \end{align} and
            \begin{align} \nonumber
                \sum_{i=1}^{n} (f(\xv) - f_{i}(\xv^{(i)})) \leq a f(\xv) + b.
        \end{align}
\end{definition}

\begin{remark} 
The following conditions hold:
\begin{enumerate} 
    \item If $M \leq 1$, then $(a,b)$ self-bounding $\implies (M,a,b)$ self-bounding.
    \item If $M > 1,$ then $(a,b)$ self-bounding $\notimplies (M,a,b)$ self-bounding.
\end{enumerate}
\end{remark}

\begin{theorem} \label{theorem:old_results_upper_tail}
Let $\Xv = (X_1,\dots,X_n)$ be a vector of independent random variables, each taking values in the set $\mathcal{X}$ and let $f: \mathcal{X}^n \to \mathbb{R}$ be a non-negative, $(a,b)$ self-bounding, and measurable function such that $Z = f(\Xv)$ has finite mean.
    \begin{itemize}
        \item For $\lambda\geq0$, it holds that \cite{mcdiarmidConcentrationSelfboundingFunctions2006}: 
        \begin{align} \nonumber
            \log \E[e^{\lambda(Z-\E[Z])}] \leq \dfrac{a \E[Z] + b}{2(1 - a\lambda)} \lambda^2, \ \ \text{for} \ 0 \leq \lambda < \frac{1}{a},
        \end{align}
        and for $t > 0$,
        \begin{align} \nonumber
            \mathbb{P}(Z - \E[Z] \geq t) \leq \exp\left(- \frac{t^2}{2(a \E[Z] + b + at)}\right).
        \end{align}
        \item (A stronger condition): Define $c = (3a - 1)/6$, then for $\lambda\geq0$, it holds that \cite{boucheronConcentrationSelfboundingFunctions2009}
        \begin{align} \nonumber
            \log \E[e^{\lambda(Z-\E[Z])}] \leq \dfrac{a \E[Z] + b}{2(1 - c_{+}\lambda)} \lambda^2,
        \end{align}
         and for all $t > 0$,
          \begin{align} \nonumber
            \mathbb{P}(Z - \E[Z] \geq t) \leq \exp\left(- \frac{t^2}{2(a \E[Z] + b + c_{+} t)}\right), 
        \end{align}
        where $c_{+} = \max\{c,0\}$.
	\end{itemize}
\end{theorem}

\begin{theorem}
\label{theorem:old_results_lower_tail}
	Let $\Xv = (X_1,\dots,X_n)$ be a vector of independent random variables, each taking values in the set $\mathcal{X}$ and let $f: \mathcal{X}^n \to \mathbb{R}$ be a non-negative, $(a,b)$ self-bounding, measurable function such that $Z = f(\Xv)$ has finite mean. For $t > 0$ it holds that \cite{mcdiarmidConcentrationSelfboundingFunctions2006}: 
    \begin{align} \nonumber
        \mathbb{P}(Z - \E[Z] \leq - t) \leq \exp \left(- \dfrac{t^2}{2(a\E[Z] + b + t/3)}\right).
    \end{align}
\end{theorem}
\section{Main results}
\begin{theorem}[Symmetry of concentration of ${(M,a,b)}$ self-bounding functions] 
Let $\Xv = (X_1,\dots,X_n)$ be a vector of independent random variables, each taking values in the set $\mathcal{X}$ and let $f: \mathcal{X}^n \to \mathbb{R}$ be a non-negative, $(M,a,b)$ self-bounding, and measurable function such that $Z = f(\Xv)$ has finite mean. Then for $\lambda<\frac{2}{aM}$, it follows that \label{theorem:M_symmetry}
\begin{align}\label{eq:cumulant_bound_symmetry}
             \log \E[e^{\lambda(Z-\E[Z])}] & \leq \dfrac{(a \E[Z] + b) M \lambda^2}{2\left(1 - \frac{a M\lambda}{2}\right)}.
\end{align}
Additionally, the following holds: 
\begin{enumerate}[label=(\alph*)]
	\item For all $t > 0$,
        \begin{align} \nonumber
            \mathbb{P}(Z - \E[Z] \geq t) \leq \exp\left(- \frac{1}{M} \frac{t^2}{2(a \E[Z] + b) + at}\right). 
        \end{align}
	\item For all $0 < t \leq \E[Z]$, 
		\begin{align} \nonumber
			\mathbb{P}(Z - \E[Z] \leq -t) \leq \exp \left(- \dfrac{1}{M} \dfrac{t^2}{2(a\E[Z] + b) + at}\right).
		\end{align}
	\end{enumerate}
\end{theorem}

The proof of Theorem \ref{theorem:M_symmetry} relies on an updated inequality bound shown in Lemma \ref{lemma:alpha_increasing}. To prove the new lower bound, we leverage Theorem \ref{theorem:Harris_inequality}. Previous approaches in the literature use only one part of Harris's inequality (typically for the upper tail); however, we leverage both inequalities to obtain both tail results.

\begin{figure}[tbh!]
    \centering
    \includegraphics[width=0.75\linewidth]{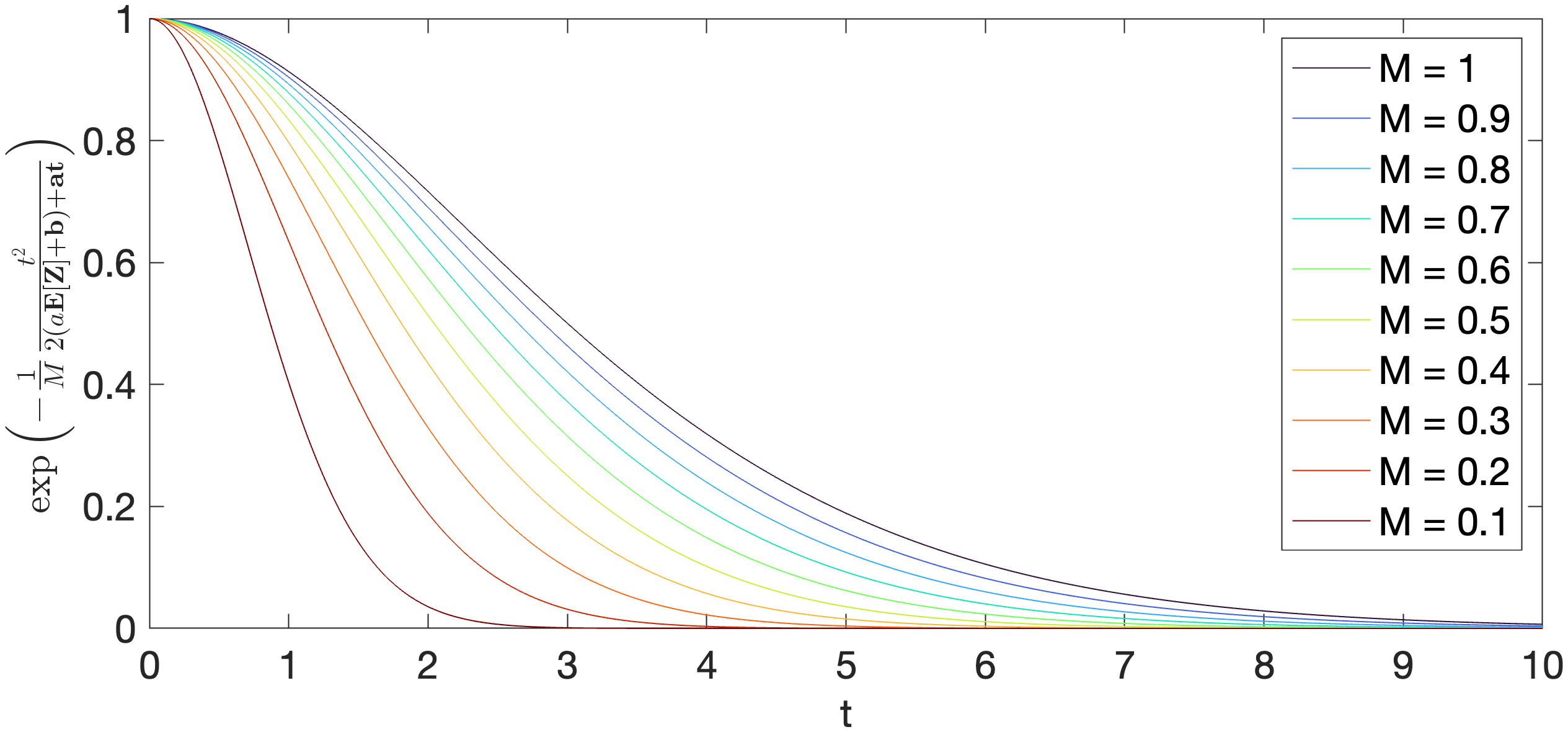}
    \caption{The bounds from Theorem \ref{theorem:M_symmetry} as shown with varying $M$, applied to the example $a=1, b = 0$ and $\E[Z] = 5$.}
    \label{fig:bounds_example}
\end{figure}

\begin{theorem}[Improved concentration of ${(M,a,b)}$ self-bounding functions] \label{theorem:M_improved}
Assume the assumptions for Theorem \ref{theorem:M_symmetry} hold. Define the functions $\delta(a,M)_+$ and $\delta(a,M)_-$ as
\begin{align} \label{eq:deltas}
    \delta(a,M)_+ \coloneq \begin{cases} 
                    0, & a \leq \frac{1}{3}, \\
                    a - \frac{1}{3}, & \frac{1}{2} < M \leq 1, \\
                    a - \frac{1}{3}, & M > 1 \ \text{and} \ a < \frac{M}{3(M-1)}, \\ 
                    a, & \text{otherwise.}
                \end{cases}, \quad  \delta(a,M)_- \coloneq \begin{cases} 
                    0, & a \geq \frac{1}{3}, \\ 
                    a, & \text{otherwise.}
                \end{cases}
\end{align}
Then the following holds:
\begin{enumerate}[label=(\alph*)] 
	\item For all $t > 0$,
        \begin{align} \nonumber
            \mathbb{P}(Z - \E[Z] \geq t) \leq \exp\left(- \frac{1}{M} \frac{t^2}{2(a \E[Z] + b) + \delta(a,M)_+ t}\right). 
        \end{align}
	\item For all $0 < t \leq \E[Z]$, 
		\begin{align} \nonumber
			\mathbb{P}(Z - \E[Z] \leq -t) \leq \exp \left(- \dfrac{1}{M} \dfrac{t^2}{2(a\E[Z] + b) + \delta(a,M)_-t}\right).
		\end{align}
\end{enumerate}
\end{theorem}

\begin{remark} \label{remark:t_delta}
    For $0 < M \leq \frac{1}{2}$, $\delta(a,M)_+$ can be strengthened such that for any $a > \frac{1}{3}, \ \delta(a,M)_+ \coloneq a-\frac{1}{3}$, but the inequality only holds for
    \begin{align}
        t \leq \dfrac{a\E[Z]+b}{a-\frac{1}{3}}\left(\left( \dfrac{\sqrt{a^2M\left(1-M+\dfrac{2M}{3a} - \dfrac{1}{3a}\right)} + a(M-1)}{M(a-\frac{1}{3})-a}\right)^{-2}-1\right).
    \end{align}
    Note that as $M \to 0,$ the bound decreases to $0$.
\end{remark}

\begin{remark}
    In the case that $a < \frac{1}{3}$, then the result in Theorem \ref{theorem:M_symmetry} improves on Theorem \ref{theorem:old_results_lower_tail}, since their coefficient of $t$ is $\frac{2}{3}$ and our result holds for $a < \frac{1}{3}$.
\end{remark}

\section{Scaling} \label{section:scaling}

To validate the result for $M>1$, we show that there is indeed improved concentration in most instances for the upper tail, when compared to rescaling an $(M,a,b)$ self-bounding function to an $(a,b)$ self-bounding function. Let $f(\xv)$ be an $(M,a_1,b_1)$ self-bounding function and set $g(\xv) \coloneq \frac{f(\xv)}{M}$ so that $g(\xv)$ is $(a_2,b_2)$ self-bounding, i.e. it satisfies both
\begin{align} \nonumber
                0 \leq g(\xv) - g_{i}(\xv^{(i)}) & \leq 1, \\ \nonumber
                \sum_{i=1}^{n} (g(\xv) - g_{i}(\xv^{(i)})) & \leq a_2 g(\xv) + b_2.
\end{align} 
Multiplying the bottom equation by $M$ implies that $a_2 = a_1$ and $b_2 = M b_1$, so that $g$ is $(a_1,Mb_1)$ self-bounding. It follows that $\E[g(\Xv)] = \frac{1}{M} \E[Z]$. Hence, by Theorem \ref{theorem:old_results_upper_tail} with $M=1$ for $g$, it follows that
 \begin{align}\mathbb{P}\left(g(\Xv) - \E[g(\Xv)] \geq t \right) = \mathbb{P}\left(\frac{Z}{M} - \frac{\E[Z]}{M} \geq t\right)
			& \leq \exp\left(- \frac{t^2}{2(a_2 \frac{\E[Z]}{M} + b_2 + c_+ t)}\right) \\ 
			& = \exp\left(- \frac{t^2}{2(a_1 \frac{\E[Z]}{M} + Mb_1 + c_+ t)}\right).
        \end{align}
By applying Theorem \ref{theorem:M_improved},
\begin{align}
	\mathbb{P}\left(\frac{Z}{M} - \frac{\E[Z]}{M} \geq t \right) = \mathbb{P}(Z - \E[Z] \geq Mt)  & \leq \exp\left(- \frac{1}{M} \frac{M^2t^2}{2(a_1 \E[Z] + b_1) + \delta(a,M)_+ Mt}\right) \\ & = \exp\left(- \frac{M t^2}{2(a_1 \E[Z] + b_1) + \delta(a,M)_+Mt}\right) \\ & = \exp\left(- \frac{t^2}{\frac{2}{M}(a_1 \E[Z] + b_1) + \delta(a,M)_+ t}\right). 
\end{align}
It then remains to show that for $M>1$, which of the following expressions obtained from both self-bounding results obtains better bounds, i.e.
\begin{align}
	& \argmin_{M}\left(\exp\left(- \frac{t^2}{2(a_1 \frac{\E[Z]}{M} + Mb_1 + c_+ t)}\right), \exp\left(- \frac{t^2}{\frac{2}{M}(a_1 \E[Z] + b_1) + \delta(a,M)_+ t}\right)\right) \\ & = \argmax_{M} \left(\frac{t^2}{2(a_1 \frac{\E[Z]}{M} + Mb_1 + c_+ t)}, \frac{t^2}{\frac{2}{M}(a_1 \E[Z] + b_1) + \delta(a,M)_+ t} \right) \\ & = \argmin_{M} \left(2(a_1 \frac{\E[Z]}{M} + Mb_1 + c_+ t), \frac{2}{M}(a_1 \E[Z] + b_1) + \delta(a,M)_+ t \right) \\ & = \argmin_{M} \left(2Mb_1 + 2c_+ t, \frac{2b_1}{M} + \delta(a,M)_+ t \right).
\end{align}
    Assume $a_1 > \frac{M}{3(M-1)}$ such that $2c_+ = a_1-\frac{1}{3}$ and $\delta(a_1,M)_{+} =a_1$. Then, for
\begin{align}
    t \leq \dfrac{2b_1}{3}\left(M - \dfrac{1}{M}\right),
\end{align}
the bound in Theorem \ref{theorem:M_improved} is tighter. If $b_1 = 0$, then scaling exhibits the additional term $-t/3$. When $\delta(a,M)_+ = 2c_+$, the bound in Theorem \ref{theorem:M_improved} is always less than the scaling result when $b_1 > 0$, or equal to the scaling result when $b_1 = 0$. For the lower tail, it is evident that Theorem \ref{theorem:M_improved} obtains better concentration than Theorem \ref{theorem:old_results_lower_tail} even when rescaling.

\section{Proofs}
\subsection{Proof of Theorem \ref{theorem:M_symmetry}}

The following results are needed for the proof of Theorem \ref{theorem:M_symmetry}.

\begin{theorem}[A modified Logarithmic Sobolev inequality \cite{boucheronConcentrationInequalitiesNonasymptotic2013}] \label{theorem:sobolev}
	Let $\Xv = (X_1,\dots,X_n)$ be a vector of independent random variables, each taking values in some measurable space $\mathcal{X}$. Let $f: \mathcal{X}^n \to \mathbb{R}$ be measurable and let $Z = f(\Xv)$. Let $\Xv^{(i)} = (X_1,\dots,X_{i-1},X_{i+1},\dots,X_n)$ and let $Z_i$ denote a measurable function of $\Xv^{(i)}$. Define $\psi(x) = e^{x} - x - 1$. Then for any $\lambda \in \mathbb{R}$,
	\begin{align} \label{eq:sobolev_eq}
		\lambda \E[Z e^{\lambda Z}] - \E[e^{\lambda Z}] \log \E[e^{\lambda Z}] \leq \sum_{i=1}^{n} \E[e^{\lambda Z} \psi (-\lambda(Z -  Z_i))].
	\end{align}
\end{theorem}
In the following, we set $Z_i \coloneq f_{i}(\Xv^{(i)})$ as in (\ref{eq:inf}).

\begin{lemma} \label{lemma:alpha_increasing}
Define $\alpha (x,\lambda) \coloneq \frac{\psi(-\lambda x)}{x \psi(-\lambda)}$ for $x \in (0,M]$. Then $\alpha(x,\lambda) \leq \alpha(M,\lambda)$ for $\lambda \in \mathbb{R} \setminus\{0\}$.
\end{lemma}

\begin{proof}
It follows that the partial derivative of $\alpha(x,\lambda)$ for any fixed $\lambda \neq 0$ is given by
\begin{align} \dfrac{\, \partial \alpha(x,\lambda)}{\, \partial x} = \dfrac{(e^{\lambda x} - \lambda x - 1)e^{\lambda - \lambda x}}{((\lambda -1)e^{\lambda} + 1)x^2}.  \end{align}
Note that $e^{\lambda x} - \lambda x - 1$ is convex, since the second derivative = $\lambda^2 e^{\lambda x} \geq 0$, and is increasing for $x\geq0$. Similarly $(\lambda -1)e^{\lambda} + 1 > 0$ for all $\lambda \in \mathbb{R}\setminus\{0\}$ since its minimum is at $\lambda = 0$ which gives the value $0$. Since all terms are positive, then $\alpha(x,\lambda)$ is an increasing function in $x$, and the result follows. This concludes the proof.    
\end{proof}

\begin{remark} \label{remark:alpha_M}
Assume $Z - Z_i \leq M$. It then follows from Lemma \ref{lemma:alpha_increasing} that for $x \in (0,M]$,
\begin{align} \alpha (x,\lambda) \leq \alpha (M,\lambda) \implies 
\label{eq:alpha_inequality} \psi(-\lambda x) \leq x \frac{\psi(-\lambda M)}{M}. \end{align}
\end{remark}

\begin{theorem}[Harris's inequality \cite{boucheronConcentrationInequalitiesNonasymptotic2013}] \label{theorem:Harris_inequality} Let $f,g : \mathbb{R}^n \to \mathbb{R}$ be nondecreasing real-valued functions. Let $X_1,\dots,X_n$ be independent real-valued random variables and let $ \Xv \coloneq (X_1,\dots,X_n)$ be the vector of $n$ random variables taking value in $\mathbb{R}^n$. Then
	\begin{align}
		\E[f(\Xv) g(\Xv)] \geq \E[f(\Xv)] \E[g(\Xv)].
	\end{align}
	If $f$ is nonincreasing and $g$ is nondecreasing, then 
	\begin{align}
		\E[f(\Xv) g(\Xv)] \leq \E[f(\Xv)] \E[g(\Xv)].
	\end{align}
\end{theorem}

The following observation is a consequence of Theorem \ref{theorem:Harris_inequality}. 

\begin{corollary} \label{corollary:cumulant_props} Define $G(\lambda) \coloneq \log \E[e^{\lambda(Z - \E(Z))}]$, and its derivative as
	\begin{align}
		G'(\lambda) = \dfrac{\E((Z - \E[Z])e^{\lambda (Z - \E[Z])})}{\E(e^{\lambda (Z - \E[Z])})}.
	\end{align}
	Then, the following holds:
	\begin{enumerate}[label=(\alph*)]
		\item $G(\lambda) \geq 0$ for all $\lambda \in \mathbb{R}$.
		\item $G'(\lambda) \geq 0$ for $\lambda > 0$.
		\item $G'(\lambda) \leq 0$ for $\lambda < 0$.
	\end{enumerate} 
\end{corollary}
\begin{proof}
	\begin{enumerate}[label=(\alph*)]
		\item[] 
        \!\!\!\!\!\!\!\!\!\!\!\!\!\!\!\!\!\!\!\!To prove ($a$), we note that by Jensen's inequality for convex functions $f$, it follows that
		\begin{align} \E[f(X)] \geq f(\E[X]). \end{align}
		Set $X = \lambda(Z - \E[Z])$ and $f(x) = e^{x}$, it then follows that
		\begin{align} \E[e^{\lambda(Z - \E[Z])}] \geq e^{\E[\lambda(Z - \E[Z])]} = e^{0} = 1. \end{align}
		By taking the logarithm on both sides, the first result follows.
		\item[] To prove ($b$), let $f(x) = e^{\lambda x}$ and $g(x) = x$. Then $f(x)$ is nondecreasing for $\lambda > 0$ and $g(x)$ is also nondecreasing. Then, by Theorem \ref{theorem:Harris_inequality}, it follows that
		\begin{align}
        \E[Ze^{\lambda Z}] - \E[Z] \E[e^{\lambda Z}] \geq 0,
		\end{align}
		and hence,
		\begin{align}
		G'(\lambda) & = \dfrac{\E[(Z - \E[Z])e^{\lambda (Z - \E[Z])}]}{\E[e^{\lambda (Z - \E[Z])}]} \\ & = \dfrac{e^{- \lambda \E[Z]}(\E[Ze^{\lambda Z}] - \E[Z]\E[e^{\lambda Z}])}{\E[e^{\lambda (Z - \E[Z])}]} \\ & \geq 0.
	\end{align}
		\item[] To prove (c), let $f$ and $g$ be as defined previously. For $\lambda < 0$, $f(x)$ is nonincreasing. Similarly by Theorem \ref{theorem:Harris_inequality}, it follows that
		\begin{align}
        \E[Ze^{\lambda Z}] - \E[Z] \E[e^{\lambda Z}] \leq 0,
		\end{align}
		and the result follows. This concludes the proof.
	\end{enumerate}
\end{proof}

\begin{theorem} \label{theorem:optimisation_symmetry} Let $s(x) \coloneq 1 + x - \sqrt{1+2x}$. The solution to the optimisation problem
    \begin{align}\label{eq:optimisation_upper}
        \sup_{0<\lambda< \frac{2}{aM}} t\lambda - \dfrac{(a \E[Z] + b) M \lambda^2}{2(1 - \frac{a M\lambda}{2})}
    \end{align}
    is given by
    \begin{align}
        \dfrac{4(a\E[Z]+b)}{a^2 M}s\left(\frac{at}{2(a\E[Z]+b)}\right),
    \end{align}
    where the value $\lambda$ that maximizes the supremum is given by
    \begin{align}\label{eq:optimisation_upper_lambda}
        \lambda = \dfrac{2}{aM}\left(1 - \left(1 + \dfrac{a t}{a\E[Z]+b}\right)^{-0.5}\right).
    \end{align}
Conversely, for $t \leq \E[Z]$, the solution to the optimisation problem
    \begin{align}\label{eq:optimisation_lower}
        \sup_{\lambda < 0} -t\lambda - \dfrac{(a \E[Z] + b) M \lambda^2}{2(1 - \frac{a M\lambda}{2})}
    \end{align}
    is given by
    \begin{align}
        \dfrac{4(a\E[Z]+b)}{a^2 M}s\left(-\frac{at}{2(a\E[Z]+b)}\right),
    \end{align}
    where the value of $\lambda$ that maximizes the supremum is given by\begin{align}\label{eq:optimisation_lower_lambda}
        \lambda^* = \dfrac{2}{aM}\left(1 - \left(1 - \dfrac{a t}{a\E[Z]+b}\right)^{-0.5}\right).
    \end{align}    
\end{theorem}
Note that the proof of this result follows similarly to \cite[Lemma 11]{boucheronConcentrationInequalitiesUsing2003}.

\begin{proof}[Proof of Theorem \ref{theorem:M_symmetry}]
Combining Theorem \ref{theorem:sobolev} with Definition \ref{def:(M,a,b)self-bounding} and Remark \ref{remark:alpha_M}, it follows that for $\lambda \in \mathbb{R}$,
\begin{align} 
	\sum_{i=1}^{n} \E[e^{\lambda Z} \psi (-\lambda(Z -  Z_i))] < \frac{\psi (-M \lambda)}{M} \sum_{i=1}^{n} \E[(Z -  Z_i) e^{\lambda Z}] & = \frac{\psi (-M \lambda)}{M} \E[e^{\lambda Z} \sum_{i=1}^{n} (Z -  Z_i)] \\ & \leq  \frac{\psi (-M \lambda)}{M} \E[e^{\lambda Z} (a Z + b)].
\end{align}
Define $G(\lambda)$ as in Corollary \ref{corollary:cumulant_props}, after some algebraic manipulation, we obtain the differential inequality 
\begin{align} \label{eq:differential_ineq}
	\left(\lambda - a \frac{\psi (-M \lambda)}{M}\right) G'(\lambda) - G(\lambda) \leq (a \E[Z] + b) \frac{\psi (-M \lambda)}{M}. 
\end{align}
We proceed by exploring the cases $\lambda > 0$ and $\lambda < 0$ separately.
\begin{enumerate}
    \item Case $\lambda > 0$. \newline
Since $\psi(-x) \leq \frac{x^2}{2}$ for $x>0$ \cite{boucheronConcentrationSelfboundingFunctions2009}, by setting $x = \lambda M$, and dividing by $\lambda^2$, (\ref{eq:differential_ineq}) yields
\begin{align} \label{eq:diff_thm_1}
	\left(\dfrac{1}{\lambda} - \frac{aM}{2}\right) G'(\lambda) - \dfrac{1}{\lambda^2}G(\lambda) \leq (a \E[Z] + b)\frac{M}{2}, \ \text{for} \ \lambda \in \left(0,\frac{2}{aM}\right). 
\end{align}
Following the same derivation approach as in \cite{boucheronConcentrationSelfboundingFunctions2009} results in
\begin{align}
	\dfrac{d}{d\lambda}\left(\left(\frac{1}{\lambda} - \frac{aM}{2}\right)G(\lambda)\right) = \left(\dfrac{1}{\lambda} - \frac{aM}{2}\right) G'(\lambda) - \dfrac{1}{\lambda^2}G(\lambda).
\end{align}
By observing that $G(0) = G'(0) = 0$, it follows by application of L'H\^opital's rule that
\begin{align}
	\lim_{s \to 0} \frac{1}{s} G(s) = 0.
\end{align}
Hence, integration of (\ref{eq:diff_thm_1}) yields 
\begin{align} 
G(\lambda) & \leq \dfrac{(a \E[Z] + b) \lambda^2 M}{2(1 - \frac{aM \lambda}{2} )}, \ \text{for} \ \lambda \in \left(0,\frac{2}{aM}\right).
\end{align}
This concludes the proof for case $\lambda>0$. 
\item Case $\lambda < 0$. \newline
From Corollary \ref{corollary:cumulant_props}, we note that for $\lambda < 0, G'(\lambda) < 0$. Using the inequality $\psi(-x) \geq \frac{x^2}{2}$ for $x < 0$, it follows that
\begin{align}
	\lambda > \lambda - \frac{a(-\lambda M)^2}{2M} \geq  \lambda - a \frac{\psi(-\lambda M)}{M},
\end{align}
multiplying by $G'(\lambda) < 0$ yields
\begin{align} \label{eq:lb_diff_ineq}
	G'(\lambda )\lambda < G'(\lambda)\left(\lambda - \frac{a(-\lambda M)^2}{2M}\right) \leq  G'(\lambda)\left(\lambda - \frac{a\psi(-\lambda M)}{M}\right).
\end{align} 
Combining (\ref{eq:differential_ineq}) with (\ref{eq:lb_diff_ineq}) , it follows that for $\lambda < 0$
\begin{align}
	G'(\lambda)\left(\lambda -\frac{a\lambda^2 M}{2}\right) - G(\lambda) \leq (a \E[Z] + b)\frac{\psi(-\lambda M)}{M}, \ \text{for} \ \lambda < 0.
\end{align}
Dividing by $\lambda^2$ and integrating from $\lambda$ to $0$ yields
\begin{align}
	-\left(\frac{1}{\lambda} - \frac{aM}{2}\right)G(\lambda) & \leq (a\E[Z] + b) \int_{\lambda}^{0} \frac{\psi(-s M)}{s^2 M} \, ds \\ & = (a\E[Z] + b) \int_{0}^{\lambda} - \frac{\psi(-s M)}{s^2 M} \, ds.
\end{align}
Using $- \psi(-x) \leq - \frac{x^2}{2}$ for $x \leq 0$, 
\begin{align}
	(a\E[Z] + b) \int_{0}^{\lambda} - \frac{\psi(-s M)}{s^2 M} \, ds \leq (a\E[Z] + b) \int_{0}^{\lambda} - \frac{M}{2} \, ds = -(a\E[Z] + b)\frac{\lambda M}{2},
\end{align}
which yields the differential inequality
\begin{align}
	-\left(\frac{1}{\lambda} - \frac{aM}{2}\right)G(\lambda) \leq -(a\E[Z] + b)\frac{\lambda M}{2}.
\end{align}
Since $\lambda < 0$, $\left(\frac{1}{\lambda} - \frac{aM}{2}\right) < 0$ for $a \geq 0, M > 0$ and $-(\frac{1}{\lambda} - \frac{aM}{2}) > 0$ for $\lambda < 0$. Hence, we obtain the bound
\begin{align}
	G(\lambda) \leq \frac{(a \E[Z] + b)\lambda^2 M}{2(1-\frac{a\lambda M}{2})}, \ \text{for} \ \lambda < 0,
\end{align} 
and the cumulant bound given in (\ref{eq:cumulant_bound_symmetry}) follows. Combining Theorem \ref{theorem:optimisation_symmetry} with the inequality $2s(x) \geq y(x)\coloneq \frac{x^2}{1+x}$ for $x \geq 0$ \cite[p. 878]{pascalmassartConstantsTalagrandsConcentration2000}, it follows that
\begin{align}
		\exp\left(-\dfrac{4(a\E[Z]+b)}{a^2 M}s\left(\frac{at}{2(a\E[Z]+b)}\right)\right) & \leq \exp\left(-\dfrac{2(a\E[Z]+b)}{a^2 M}y\left(\frac{at}{2(a\E[Z]+b)}\right)\right) \\ & = \exp\left(-\dfrac{1}{M} \dfrac{t^2}{2(a\E[Z] + b) + at}\right).
\end{align}
Similarly, application of the inequality $2s(-x) \geq y(x)$ for $0 \leq x \leq \frac{1}{2}$ yields
	\begin{align}
		\exp\left(-\dfrac{4(a\E[Z]+b)}{a^2 M}s\left(-\frac{at}{2(a\E[Z]+b)}\right)\right) & \leq \exp\left(-\dfrac{2(a\E[Z]+b)}{a^2 M}y\left(\frac{at}{2(a\E[Z]+b)}\right)\right) \\ & = \exp\left(-\dfrac{1}{M} \dfrac{t^2}{2(a\E[Z] + b) + at)}\right).
	\end{align}	
This holds for $0 \leq \frac{a t}{2(a\E[Z]+b)} \leq \frac{1}{2}$ which holds for $t \leq \E[Z] + \frac{b}{a}$. To prove the inequality $2s(-x) \geq y(x)$ for a specific value of $x$, we obtain the Taylor expansion around $x=0$, and the result follows.	This concludes the proof for the case $\lambda<0$.
\end{enumerate}
\end{proof}

\subsection{Proof of Theorem \ref{theorem:M_improved}}

The proof of Theorem \ref{theorem:M_improved} follows a similar structure to that in \cite{boucheronConcentrationSelfboundingFunctions2009}, with some minor amendments. 

\begin{lemma} \label{lemma:Boucheron_ODE} Let $f$ be a non-decreasing continuosly differentiable function on an open interval $I$ containing $0$ such that $f(0)=0, f'(0) > 0$ and $f'(x) \neq 0$ for $x \neq 0$. Let also $g$ be a continuous function on $I$ and consider an infinitely many times differentiable function $G$ on $I$ such that $G(0) = G'(0) = 0$ and for every $\lambda \in I$,
\begin{align}
f(\lambda)G'(\lambda) - f'(\lambda)G(\lambda) \leq f^2(\lambda) g(\lambda). \nonumber
\end{align}
Then, for every $\lambda \in I, G(\lambda) \leq f(\lambda) \int_{0}^\lambda g(x)\, dx$.    
\end{lemma}
\begin{proof}
The proof of this result can be found in \cite[Lemma 3]{boucheronConcentrationSelfboundingFunctions2009}.
\end{proof}

\begin{lemma} \label{lemma:boucheron_second_proof} Let $I$ be an open interval containing $0$ and let $\rho$ be continuous on $I$. Let also $a \geq 0$ and $v>0$. Let $H: I \to \mathbb{R}$, be an infinitely many times differentiable function satisfying
\begin{align}
    M (\lambda H'(\lambda) - H(\lambda)) \leq \rho(\lambda)(a H'(\lambda) + a \E[Z] + b)
\end{align}
with
\begin{align} \label{eq:lemma_4_condition_1}
    aH'(\lambda) + a \E[Z] + b > 0, \quad \text{for} \ \lambda \in I \ \text{and} \ H'(0) = H(0) = 0.
\end{align}
Consider the function $\rho_0: I \to \mathbb{R}$. Assume that $G_0: I \to \mathbb{R}$ is infinitely many times differentiable such that for $\lambda \in I$,
\begin{align}
    1 + aG'_{0}(\lambda)> 0, \quad \text{and} \quad \ G_0(0) = G_0'(0) = 0, \ \quad\text{and}\quad \ G''_{0}(0) = M.
\end{align}
Assume that $G_0$ solves the differential equation
\begin{align}
    M(\lambda G'_{0}(\lambda) - G_0(\lambda)) = \rho_{0}(\lambda)(a G'_0(\lambda) + 1).
\end{align}
Then if $\rho(\lambda) \leq \rho_0(\lambda)$ for $\lambda \in I$, it holds that $H \leq v G_0$.
\end{lemma}

\begin{proof} There are some differences between this lemma and \cite[Lemma 4]{boucheronConcentrationSelfboundingFunctions2009}, so whilst the proof structure is similar, for the sake of completeness, we present it in full. Using the assumption that $\rho(\lambda) \leq \rho_{0}(\lambda)$ and the rest of the assumptions presented in the statement of the lemma, it follows that
\begin{align}
M (\lambda H'(\lambda) - H(\lambda)) \leq \dfrac{M(\lambda G'_{0}(\lambda) - G_0(\lambda))}{aG'_0(\lambda) + 1}(a H'(\lambda) + v),
\end{align}
and thus, expressing the inequality in terms similar to Lemma \ref{lemma:Boucheron_ODE} yields
\begin{align}
    (\lambda + a G_0(\lambda)) H'(\lambda) - (1+aG'_{0}(\lambda)) H(\lambda) \leq (a \E[Z] + b) (\lambda G'_0(\lambda) - G_0(\lambda)).
\end{align}
Define the functions $f(\lambda) = \lambda + a G_0(\lambda)$ for $\lambda \in I$ and $g:I\to \mathbb{R}$ as the piecewise function
\begin{align}
    g(\lambda) = \begin{cases}
    (a\E[Z] + b) \dfrac{(\lambda G'_{0}(\lambda) - G_0(\lambda))}{(\lambda + aG_0'(\lambda))^2}, & \lambda \neq 0, \vspace{2mm} \\
    (a\E[Z] + b) \dfrac{M}{2}, & \lambda = 0.
    \end{cases}
\end{align}
Then the function $g(\lambda)$ is continuous on $I$, and by Lemma \ref{lemma:Boucheron_ODE}, it holds that
\begin{align}
    H(\lambda) \leq f(\lambda) \int_{0}^{\lambda} g(x) \, dx = (a\E[Z]+b) f(\lambda) \int_{0}^\lambda \left( \dfrac{G_0(x)}{f(x)} \right)' \, dx,
\end{align}
and the result follows. This concludes the proof.
\end{proof}
\begin{remark}
To obtain the cumulant bound in Theorem \ref{theorem:M_symmetry} for $\lambda > 0$, note that the solution to the differential equation, parametrized by $\gamma \geq 0$,
\begin{align}
    M(\lambda G'_{\gamma}(\lambda) - G_{\gamma}(\lambda)) = \dfrac{(-\lambda M)^2}{2}(\gamma G'_{\gamma}(\lambda) + 1),
\end{align}
is given by
\begin{align} \label{eq:g_gamma_sol}
    G_{\gamma}(\lambda) = \dfrac{M \lambda^2}{2\left(1-\frac{\gamma \lambda M}{2}\right)}, \quad \frac{\gamma \lambda M}{2} \neq 1,
\end{align}
where we have used the fact that $\rho(\lambda) = \psi(-\lambda M) \leq \rho_0(\lambda) = \frac{(-\lambda M)^2}{2}$. Substituting $\gamma = a$ yields the bound. Note that the solution is consistent since for $\gamma \geq 0$, $G_{\gamma}(0) = G_{\gamma}'(0) = 0 \ \text{and} \ G''_{\gamma}(0) = M$.
\end{remark}

\subsubsection{Expressions for $\delta(a,M)_+$ and $\delta(a,M)_-$} The calculation of both quantities require two steps. We start by first checking that the assumptions for Lemma \ref{lemma:boucheron_second_proof} hold (\textit{condition} $1$). We then proceed by examining whether the solution holds for the solution of the optimisation problem posed by the application of Markov's inequality (\textit{condition} $2$). Consider Lemma \ref{lemma:boucheron_second_proof} and define the function
\begin{align}
    \rho_{\gamma}(\lambda) \coloneq \dfrac{M(\lambda G'_{\gamma}(\lambda) - G_{\gamma}(\lambda))}{1 + a G'_{\gamma}(\lambda)},
\end{align}
which, for some interval $I$, is the solution to the differential equation 
\begin{align}
    M(\lambda H'(\lambda) - H(\lambda)) = \rho_{\gamma}(\lambda)(1 + aH'(\lambda)),
\end{align}
provided that $1 + aG_{\gamma}'(\lambda)>0$ on $I$. Thus, we seek to find the smallest $\gamma$ on the relevant interval $I$ such that it satisfies the following two conditions:
\begin{align} \label{eq:condition_1}
    \phi(-\lambda M) \leq \rho_{\gamma}(\lambda), \\ 1 + a G'_{\gamma}(\lambda) > 0, \label{eq:condition_2}
\end{align}
since $G_{\gamma}(\lambda)$ solves other differential equations, in addition to (\ref{eq:g_gamma_sol}) (on the relevant interval).
To that end, introduce $D_{\gamma}(\lambda) = (M\gamma \lambda - 2)^2 (1 + a G'_{\gamma}(\lambda)) = 4(1 + \lambda M (a - \gamma) + \frac{M \gamma \lambda^2}{4}(M\gamma - a))$, then it holds that 
\begin{align} \rho_{\gamma}(\lambda) = \dfrac{2M^2 \lambda^2}{D_{\gamma}(\lambda)}, \quad \text{for} \ M \gamma \lambda \neq 2. \end{align}
Note that $1 + aG'_{\gamma}(\lambda) > 0 \iff D_{\gamma}(\lambda) > 0$. 

\subsection{Examination of satisfaction of {\it Condition 1}}
\subsubsection{Case $\lambda \geq 0$ for $\delta(a,M)_+$.}
\begin{enumerate}[label=(\alph*)] 
    \item First, set $\gamma = 0$, then $D_{0}(\lambda) = 4(1 + \lambda M a) > 0$ for $\lambda \geq 0$. For $\lambda \geq 0$, we seek to find $a$ such that
\begin{align} \rho_{0}(\lambda) = \dfrac{2M^2\lambda^2}{4(1+\lambda M a)} \geq e^{-\lambda M} + \lambda M - 1 = \phi(-\lambda M). \end{align} 
Rewriting it in terms of $a$ results in
\begin{align}
    a \leq \mu(\lambda) = \dfrac{2M^2\lambda^2 - 4(e^{-\lambda M} + M\lambda - 1)}{4 \lambda M(e^{-\lambda M} + \lambda M - 1)}.
\end{align}
Since the function $\mu$ is monotonically increasing in $\lambda \geq 0$, the the Mclaurin series expansion for $e^{-\lambda M} = 1 - \lambda M + \dfrac{(\lambda M)^2}{2} - \dfrac{(\lambda M)^3}{6} + O(\lambda M)^4$, enables us to express the fraction as
\begin{align}
    \dfrac{2M^2\lambda^2 - 4(e^{-\lambda M} + M\lambda - 1)}{4 \lambda M(e^{-\lambda M} + \lambda M - 1)} & = \dfrac{2M^2\lambda^2 - 4\left(\dfrac{(\lambda M)^2}{2} - \dfrac{(\lambda M)^3}{6} + O(\lambda M)^4\right)}{4 \lambda M\left(\dfrac{(\lambda M)^2}{2} - \dfrac{(\lambda M)^3}{6} + O(\lambda M)^4\right)} \\ & = \dfrac{\dfrac{2(\lambda M)^3}{3} + O(\lambda M)^4}{2(\lambda M)^3 + O(\lambda M)^4} \to \dfrac{1}{3} \ \text{as} \ \lambda \to 0.
\end{align}
\item Set $\gamma = a-\frac{1}{3}$ for $a > \frac{1}{3}$, then we aim to prove that
\begin{align} D_{a-\frac{1}{3}}(\lambda) > 0 \ \text{and} \ \rho_{a-\frac{1}{3}}(\lambda) = \dfrac{2M^2\lambda^2}{D_{a - \frac{1}{3}}(\lambda)} \geq \phi(-\lambda M).\end{align}
From \cite{boucheronConcentrationSelfboundingFunctions2009}, we have the inequality
\begin{align}
    \dfrac{x^2}{2(1+\frac{x}{3})} \geq \phi(-x), \ \text{for} \ x \geq 0.
\end{align}
Since $M>0$, setting $x = \lambda M$ results in 
\begin{align}
    \dfrac{M^2 \lambda^2}{2\left(1+ \frac{M \lambda}{3}\right)} \geq \phi(-\lambda M), \ \text{for} \ \lambda \geq 0.
\end{align}
Therefore, the interval $I$ for which it holds that
\begin{align}
    \dfrac{M^2\lambda^2}{2\left(1 + \dfrac{\lambda M}{3} + \dfrac{M \left(a-\frac{1}{3}\right) \lambda^2}{4}(M\left(a-\frac{1}{3}\right) - a)\right)} \geq \dfrac{M^2 \lambda^2}{2\left(1+ \dfrac{\lambda M}{3}\right)} \geq \phi(-\lambda M),
\end{align}
depends on the condition
\begin{align} \label{eq:neg_sign}
    \dfrac{M (a-\frac{1}{3}) \lambda^2}{4}\left(M\left(a-\frac{1}{3}\right) - a\right) \leq 0 \iff M\left(a-\frac{1}{3}\right) - a \leq 0.
\end{align}
For $M > 1$, we require
\begin{align}
    \dfrac{1}{3} < a \leq \dfrac{M}{3(M-1)}.
\end{align}
For $M \leq 1$, then it follows that
\begin{align}
    M\left(a-\frac{1}{3}\right) - a < 0 \implies a(M-1) < \frac{M}{3} \implies a > \dfrac{M}{3(M-1)}.
\end{align}
Let $r(M) \coloneq \frac{M}{3(M-1)}$, then $r(0) = 0$ and $r'(M) < 0$ for $0 < M < 1$, and so for $M < 1$, it holds for any $a-\frac{1}{3} \geq 0$. When $M = 1$, $(a-\frac{1}{3}) - a < 0$ holds for any $a > \frac{1}{3}$. This holds for $I = \mathbb{R}$ since $\lambda^2 \geq 0$. We now consider the effects of $M\left(a-\frac{1}{3}\right) - a \leq 0$ on the condition $D_{a-\frac{1}{3}}(\lambda) > 0$. Under this condition, there exist two real roots of the equation
\begin{align} \label{eq:cond_1_positive}
    D_{a-\frac{1}{3}}(\lambda) = 2\left(1 + \dfrac{\lambda M}{3} + \dfrac{M \left(a-\frac{1}{3}\right) \lambda^2}{4}\left(M\left(a-\frac{1}{3}\right) - a\right)\right).    
\end{align}
Note that it has one positive and one negative root, given by 
\begin{align}
    \lambda = \dfrac{-\frac{M}{3} \pm \sqrt{\frac{M^2}{9} - M(a-\frac{1}{3})(M(a-\frac{1}{3})-a)}}{\frac{M}{2}(a-\frac{1}{3})(M(a-\frac{1}{3})-a)}.
\end{align}
The positive root, for $M\left(a-\frac{1}{3}\right) - a < 0$, is given by 
\begin{align}
    \lambda^* = \dfrac{-\frac{M}{3} - \sqrt{\frac{M^2}{9} - M(a-\frac{1}{3})(M(a-\frac{1}{3})-a)}}{\frac{M}{2}(a-\frac{1}{3})(M(a-\frac{1}{3})-a)}.
\end{align}
All that remains is to notice that $D''_{a-\frac{1}{3}}(\lambda) < 0$, so that $D_{a-\frac{1}{3}}(\lambda) > 0$ for $\lambda \in [0,\lambda^*)$.
\end{enumerate}
\subsubsection{Case $\lambda \leq 0$ for $\delta(a,M)_-$.}
Before proceeding, we make note that condition (\ref{eq:lemma_4_condition_1}) of Lemma \ref{lemma:boucheron_second_proof} is always satisfied by (\ref{eq:differential_ineq}), even for $\lambda < 0$.
\begin{enumerate}[label=(\alph*)]
\item $D_{0}(\lambda) > 0$ holds for
\begin{align}
    1 + \lambda M a > 0 \implies \lambda > -\dfrac{1}{aM}.
\end{align}
Then, checking for the values $a > 0$ with $\lambda < 0$ such that
\begin{align} \rho_{0}(\lambda) = \dfrac{2M^2\lambda^2}{4(1+\lambda M a)} \geq e^{-\lambda M} + \lambda M - 1 = \phi(-\lambda M). \end{align} 
And in terms of $a$, since $-\frac{1}{aM} < \lambda < 0$
\begin{align}
    a \geq \mu_2(\lambda) = \dfrac{2M^2\lambda^2 - 4(e^{-\lambda M} + M\lambda - 1)}{4 \lambda M(e^{-\lambda M} + \lambda M - 1)}.
\end{align}
On this interval, $\mu_2$ is decreasing in $\lambda$ for fixed $M$. Hence, by taking the limit as $\lambda \to 0$, we obtain $a \geq \frac{1}{3}$. Checking for the values $\lambda > -\frac{3}{M}$, the following inequality holds: 
\begin{align}
    \dfrac{M^2\lambda^2}{2(1 + \frac{\lambda M}{3})} \geq e^{-\lambda M} + \lambda M - 1. 
\end{align}
Hence, for $a \geq \frac{1}{3}$ and $-\frac{1}{aM} < \lambda < 0$,
\begin{align}
    \rho_{0}(\lambda) = \dfrac{M^2\lambda^2}{2(1+\lambda M a)} \geq \dfrac{M^2\lambda^2}{2(1 + \frac{\lambda M}{3})} \geq \phi(-\lambda M).
\end{align}
\end{enumerate}
\subsection{Examination of satisfaction of {\it Condition 2}}
\subsubsection{Case $\lambda > 0$.}
\begin{enumerate}[label=(\alph*)]
\item If $0 \leq a \leq \frac{1}{3}$ and $\delta(a,M)_+ = 0$, then it follows that 
\begin{align}
        \exp\left(- \sup_{\lambda > 0} \left(t\lambda - \dfrac{(a \E[Z] + b) M \lambda^2}{2}\right)\right)= \exp\left(-\dfrac{1}{M} \dfrac{t^2}{2(a\E[Z] + b)} \right).
    \end{align}
\item If $a > \frac{1}{3}$, $M \leq 1$ and $\delta(a,M)_+ = a-\frac{1}{3}$, then we check when $\lambda_2$, defined as
    \begin{align} \label{eq:condition_2_check_positive}
        \lambda_2 \coloneq \sup_{\lambda > 0} t\lambda - \dfrac{(a\E[Z] + b) M \lambda^2}{2\left(1 - \frac{(a-\frac{1}{3}) M\lambda}{2}\right)} = \dfrac{2}{(a-\frac{1}{3})M}\left(1 - \left(1 + \dfrac{(a-\frac{1}{3}) t}{a\E[Z] + b}\right)^{-0.5}\right),
    \end{align}
    satisfies both $\lambda_2< \lambda^*$ and $\lambda_2 < \frac{2}{(a-\frac{1}{3})M}.$ Note that $\lambda_2 < \frac{2}{(a-\frac{1}{3})M}$ for all $t \geq 0$. To check $\lambda_2< \lambda^*$, we need a more thorough examination. In this setting, $M(a-\frac{1}{3})-a < 0$, which yields in the inequality
    \begin{align}\label{eq:optimisation_2_check}
        \frac{M}{3} + \sqrt{\frac{M^2}{9} - M\left(a-\frac{1}{3}\right)\left(M\left(a-\frac{1}{3}\right)-a\right)} + \left(M\left(a - \frac{1}{3}\right)-a\right)\left(1 - \left(1 + \dfrac{(a-\frac{1}{3}) t}{v}\right)^{-0.5}\right) & > 0.
    \end{align}
    Taking $t \to \infty$ implies, after some manipulations and noting $M \leq 1$, the inequality
    \begin{align}
            M\left(1 - M + \dfrac{2M}{3a} - \dfrac{1}{3a}\right) > (M-1)^2,
        \end{align}
which is satisfied for $\frac{1}{2} < M \leq 1$. Consider now $ M > 1$ with $\frac{1}{3} < a < \frac{M}{3(M-1)}$ and $\delta(a,M)_+ = a - \frac{1}{3}$. Following the same analysis as the one presented after (\ref{eq:optimisation_2_check}), it suffices to check the condition
\begin{align}
    \dfrac{M}{3} + M\left(a-\frac{1}{3}\right)-a > 0 \iff a(M-1) > 0,
\end{align}
which follows since $M > 1$.
\end{enumerate}
\subsubsection{Case $\lambda < 0$}
If $a > \frac{1}{3}$ and $\delta(a,M)_+ = 0$, then it follows that \begin{align}
        \exp\left(- \sup_{-\frac{1}{aM}< \lambda < 0} \left(-t\lambda - \dfrac{(a \E[Z] + b) M \lambda^2}{2}\right)\right)= \exp\left(-\dfrac{1}{M} \dfrac{t^2}{2(a\E[Z] + b)} \right),
    \end{align}
    when $-\frac{1}{aM} \leq -\frac{t}{(a\E[Z]+b)M} \iff t \leq \E[Z] + \frac{b}{a}$.

\begin{proof}[Proof of Remark \ref{remark:t_delta}]
    The analysis follows from (\ref{eq:optimisation_2_check}), and by further noting that for $M \leq \frac{1}{2}, -a(M-1) > 0$, and the result follows by algebraic manipulation.
\end{proof}

\begin{remark} \label{remark:improvement_upper_tail}
For some values of $a > \frac{1}{3}$ and $M$, with $M\left(a-\frac{1}{3}\right) - a \geq 0$, there exist solutions on $I =(0,\tilde{\lambda})$ with $\tilde{\lambda}>0$, to the inequality
    \begin{align} \label{eq:remark_impovements_uppertail}\dfrac{M^2\lambda^2}{2\left(1 + \dfrac{\lambda M}{3} + \dfrac{M \left(a-\frac{1}{3}\right) \lambda^2}{4}(M\left(a-\frac{1}{3}\right) - a)\right)} 
    \geq \phi(-\lambda M).
\end{align}
\end{remark}
For example, with $a=1$ and $M=1.9$ such that $M\left(a-\frac{1}{3}\right) - a = 4/15$, (\ref{eq:remark_impovements_uppertail}) holds on $I = (0,1.1)$, which can be observed numerically. In the case that $M(a-\frac{1}{3})-a > 0$, it follows that (\ref{eq:cond_1_positive}) holds for $\lambda \geq 0$ and \textit{condition} $1$ holds. For \textit{condition} $2$, assume the conditions on $a$ and $M$ hold with $M(a-\frac{1}{3})-a > 0$ on $I = (0,\tilde{\lambda})$ from (\ref{eq:condition_2_check_positive}), it suffices to check when
\begin{align} \label{eq:condition_numerics}
    \lambda_2 = \dfrac{2}{(a-\frac{1}{3})M}\left(1 - \left(1 + \dfrac{(a-\frac{1}{3}) t}{a\E[Z] + b}\right)^{-0.5}\right) < \tilde{\lambda},
\end{align}
which holds when
\begin{align}
 \left(1 + \dfrac{(a-\frac{1}{3}) t}{a\E[Z] + b}\right)^{-0.5} > 1-\dfrac{\tilde{\lambda}}{2}\left(a-\frac{1}{3}\right)M.
\end{align}
In the case $1-\frac{\tilde{\lambda}}{2}\left(a-\frac{1}{3}\right)M \leq 0$, then the inequality is satisfied for all $t > 0$. This can be found in examples such as $a = 1, M = 1.8$ and $\tilde{\lambda} \approx 2.63$, which is a significant improvement as Theorem \ref{theorem:M_improved} holds only for $a < \frac{M}{3(M-1)} = 0.75$. In the case that $1-\frac{\tilde{\lambda}}{2}\left(a-\frac{1}{3}\right)M > 0$, then (\ref{eq:condition_numerics}) holds for
\begin{align}
    t \leq \dfrac{a\E[Z]+b}{a-\frac{1}{3}}\left(    \dfrac{1}{\left(1-\frac{\tilde{\lambda}}{2}(a-\frac{1}{3})M\right)^2} - 1\right).
\end{align}





\section*{Acknowledgement}
This research was funded by the UKRI Engineering and Physical Sciences Research Council (EPSRC grant number EP/S023666/1) through the Centre for Doctoral Training in Water Infrastructure and Resilience. The authors would also like to thank Thames Water for their contribution to the EPSRC grant number EP/S023666/1.

\bibliographystyle{plainnat}   
\bibliography{bib}             

\begin{thebibliography}{17}
\providecommand{\natexlab}[1]{#1}
\providecommand{\url}[1]{\texttt{#1}}
\expandafter\ifx\csname urlstyle\endcsname\relax
  \providecommand{\doi}[1]{doi: #1}\else
  \providecommand{\doi}{doi: \begingroup \urlstyle{rm}\Url}\fi

\bibitem[Boucheron et~al.(2000)Boucheron, Lugosi, and Massart]{boucheronSharpConcentrationInequality2000}
Stéphane Boucheron, Gábor Lugosi, and Pascal Massart.
\newblock A sharp concentration inequality with applications.
\newblock \emph{Random Structures \& Algorithms}, 16\penalty0 (3):\penalty0 277--292, 2000.

\bibitem[Boucheron et~al.(2003)Boucheron, Lugosi, and Massart]{boucheronConcentrationInequalitiesUsing2003}
Stéphane Boucheron, Gábor Lugosi, and Pascal Massart.
\newblock Concentration {{Inequalities Using}} the {{Entropy Method}}.
\newblock \emph{The Annals of Probability}, 31\penalty0 (3):\penalty0 1583--1614, 2003.

\bibitem[Boucheron et~al.(2009)Boucheron, Lugosi, and Massart]{boucheronConcentrationSelfboundingFunctions2009}
Stéphane Boucheron, Gábor Lugosi, and Pascal Massart.
\newblock On concentration of self-bounding functions.
\newblock \emph{Electronic Journal of Probability [electronic only]}, 14, 2009.

\bibitem[Boucheron et~al.(2013)Boucheron, Lugosi, and Massart]{boucheronConcentrationInequalitiesNonasymptotic2013}
Stéphane Boucheron, Gábor Lugosi, and Pascal Massart.
\newblock \emph{Concentration {{Inequalities}}: {{A Nonasymptotic Theory}} of {{Independence}}}.
\newblock Oxford University Press, 2013.

\bibitem[Crowley and Esnaola(2024)]{crowleySubmodularityMutualInformation}
George Crowley and Iñaki Esnaola.
\newblock Submodularity of {{Mutual Information}} for {{Multivariate Gaussian Sources}} with {{Additive Noise}}, 2024.
\newblock arXiv: 2409.03541.

\bibitem[Crowley et~al.(2024)Crowley, Tait, Panoutsos, Speight, and Esnaola]{crowleyInformationtheoreticSensorPlacement2024}
George Crowley, Simon Tait, George Panoutsos, Vanessa Speight, and Iñaki Esnaola.
\newblock Information-theoretic sensor placement for large sewer networks.
\newblock \emph{Water Research}, 268:\penalty0 122718, 2024.

\bibitem[Kontorovich(2014)]{kontorovichConcentrationUnboundedMetric2013}
Aryeh Kontorovich.
\newblock Concentration in unbounded metric spaces and algorithmic stability.
\newblock In \emph{Proceedings of the 31st International Conference on Machine Learning}, volume 32(2) of \emph{Proceedings of Machine Learning Research}, pages 28--36. PMLR, 2014.

\bibitem[Ledoux(1997)]{ledouxTalagrandsDeviationInequalities1997}
Michel Ledoux.
\newblock On {{Talagrand}}'s deviation inequalities for product measures.
\newblock \emph{ESAIM: Probability and Statistics}, 1:\penalty0 63--87, 1997.

\bibitem[Maurer(2006)]{maurerConcentrationInequalitiesFunctions2006}
Andreas Maurer.
\newblock Concentration inequalities for functions of independent variables.
\newblock \emph{Random Structures \& Algorithms}, 29\penalty0 (2):\penalty0 121--138, 2006.

\bibitem[Maurer(2010)]{maurerDominatedConcentration2010}
Andreas Maurer.
\newblock Dominated concentration.
\newblock \emph{Statistics \& Probability Letters}, 80\penalty0 (7):\penalty0 683--689, 2010.

\bibitem[McDiarmid(1989)]{mcdiarmidMethodBoundedDifferences1989}
Colin McDiarmid.
\newblock On the method of bounded differences.
\newblock In \emph{Surveys in {{Combinatorics}}, 1989: {{Invited Papers}} at the {{Twelfth British Combinatorial Conference}}}, London {{Mathematical Society Lecture Note Series}}, pages 148--188. Cambridge University Press, 1989.

\bibitem[McDiarmid and Reed(2006)]{mcdiarmidConcentrationSelfboundingFunctions2006}
Colin McDiarmid and Bruce Reed.
\newblock Concentration for self-bounding functions and an inequality of {Talagrand}.
\newblock \emph{Random Structures \& Algorithms}, 29\penalty0 (4):\penalty0 549--557, 2006.

\bibitem[{Pascal Massart}(2000)]{pascalmassartConstantsTalagrandsConcentration2000}
{Pascal Massart}.
\newblock About the constants in {{Talagrand}}'s concentration inequalities for empirical processes.
\newblock \emph{The Annals of Probability}, 28\penalty0 (2):\penalty0 863--884, 2000.

\bibitem[Pellegrina(2021)]{pellegrina2021sharperconvergenceboundsmonte}
Leonardo Pellegrina.
\newblock Sharper convergence bounds of monte carlo rademacher averages through self-bounding functions, 2021.
\newblock arXiv: 2010.12103.

\bibitem[Talagrand(1995)]{talagrandConcentrationMeasureIsoperimetric1995}
Michel Talagrand.
\newblock Concentration of measure and isoperimetric inequalities in product spaces.
\newblock \emph{Publications Mathématiques de l'Institut des Hautes Études Scientifiques}, 81\penalty0 (1):\penalty0 73--205, 1995.

\bibitem[Talagrand(1996)]{talagrandNewConcentrationInequalities1996}
Michel Talagrand.
\newblock New concentration inequalities in product spaces.
\newblock \emph{Inventiones mathematicae}, 126\penalty0 (3):\penalty0 505--563, 1996.

\bibitem[Vondrak(2010)]{vondrak2010noteconcentrationsubmodularfunctions}
Jan Vondrak.
\newblock A note on concentration of submodular functions, 2010.
\newblock arXiv: 1005.2791.

\end{thebibliography}

\end{document}